\documentclass[12pt]{article}
\usepackage[english]{babel}
\usepackage{amsthm,amsfonts, amsbsy, amssymb,amsmath,graphicx}
\usepackage{graphics}

\newtheorem{thm}{Theorem}
\newtheorem{cor}{Corollary}

\theoremstyle{remark}

\theoremstyle{definition}

\newtheorem{exa}{Example}

\def\Z{{\mathbb Z}}

\def\R{{\mathbb R}}

\title{Group-valued invariant of knots in the full torus}
\author{Vassily Olegovich Manturov\footnote{Moscow Institute of Physics and technology,
Kazan Federal University, Northeastern University (China)},\\
 Igor Mikhailovich Nikonov}
\date{}

\begin{document}

\maketitle



\section{Introduction}

Knot concordance plays a crucial role in the low dimensional topology~\cite{Liv}.


The better understanding of knot sliceness and the concordance group of classical knots
led to significant results~\cite{Man, MP}.


In the last decades, detecting new $\Z^{\infty}$ summands in the group $\Theta^3_\Z$
were considered as significant results~\cite{DHST}.

To tackle these problems, a very elaborated techniques was developed.
Say, Heegaard-Floer Homology originally used some count of holomorphic discs and $spin^{c}$-structures
on manifolds.

The present paper deals with the sliceness problems for knots in the full torus\footnote{Note that once forever we deal only with homologically trivial knots (having winding number $0$ with respect to $S^{1}$.}
 $S^{1}\times D^{1}$.

We propose a very elementary techniques which allows one to construct a lot of sliceness obstructions
for such knots.

Our approach deals with group theoretical techniques; it is completely combinatorial, and the groups
are very easy to deal with.

\section{Definitions}

Let $K$ be a knot in the full torus $S^1\times \R^2$. The full torus can be considered as a thickening of the cylinder $S^1\times\R^1$. Hence, the knot $K$ is presented by its diagram $D$ in the cylinder.

Consider the group

\begin{multline}\label{eq:group_G'}
G'=\left\{a,b,b',B,B'\mid a^2=1, ab=(b')^{-1}a, aB=(B')^{-1}a,\right.\\ \left. bB^{-1}=B'(b')^{-1}, b^{-1}B=(B')^{-1}b'\right\}.
\end{multline}

There is a bijection $\phi$ on $G'$ given by the formula
\[
\phi(x_1^{\alpha_1}x_2^{\alpha_2}\cdots x_n^{\alpha_n})=x_2^{-\alpha_2}\cdots x_n^{-\alpha_n}x_1^{-\alpha_1}
\]
where $x_i=a,b,b',B$ or $B'$, and $\alpha_i=\pm1$. Let $\tilde G'=G'/\phi$.

With an oriented knot $K$ in the full torus $S^1\times\R^2$ having winding number $0$ we associate an element $f'(K)\in \tilde G'$ as follows.

We fix a reference point $P$ (not a crossing) on the Gauss diagram $G(D)$ of the knot. We enumerate the chord endpoints in the core circle as they appear according to the orientation. The endpoints are called to be {\em in odd position} if the number of endpoints containing itself from the starting point of $G(K)$ is odd. Otherwise the endpoint is called to be {\em in even position}.

We start walking from the reference point along the core circle of $G(D)$ and write the letters according to the following rule:

\begin{enumerate}
  \item We associate $a$ with an even chord end;
  \item We associate $b, b', B$ or $B'$ to chord ends in odd positions and we associate $b^{-1}, (b')^{-1}, B^{-1}$ or $(B')^{-1}$ to chord ends in even position;
  \item We choose $b, b^{-1}, B$ or $B^{-1}$ if the chord considered is linked with evenly many even chords and $b', (b')^{-1}, B', (B')^{-1}$, otherwise;
  \item We associate $b, b^{-1}, b'$ or $(b')^{-1}$ to undercrossings, and $B, B^{-1}, B', (B')^{-1}$ to overcrossings.
\end{enumerate}

The product of the letters is a word $f'(K)$ considered as an element of $\tilde G'$.

\begin{thm}\label{thm1}
The element $f'(K)$ is an invariant of oriented knots $K$ in the full torus.
\end{thm}

The theorem is proved by standard checking Reidemeister moves. The bijection $\phi$ appears whenever we change the reference point.

It turns out that the invariant sees much more than just knot invariance:
\begin{thm}\label{thm2}
Let $K$ be a knot in the full torus. If $f'(K)=\tilde 1\in \tilde G'$ then K is not slice.
\end{thm}

The theorem follows from a more general statement.

Let $G$ be a group and $\Phi$ be a family of bijections  $\phi\colon G\to G$ such that $\phi(1)=1$. For example,
\[
\Phi=\{Ad_g \mid\ g\in G\}\mbox{ where } Ad_g(h)=ghg^{-1}, h\in G.
\]

Let $\tilde G=G/\Phi$ be the set of equivalence classes of elements of $G$ modulo the relations $g\sim \phi(g)$, $g\in G$, $\phi\in\Phi$. We will denote the equivalence class of an element $g\in G$ by $\tilde g$. Note that by definition $\tilde 1=\{1\}$.

Let $f$ be an invariant of oriented links with ordered components whose value on a link $L=K_1\cup\cdots\cup K_n$ is an element $f(L)=(f_i(L))_{i=1}^n\in\prod_{i=1}^n\tilde G=\tilde G^n$. Assume $f$ obeys the following conditions:

\begin{enumerate}
  \item For any permutation $\sigma\in\Sigma_n$ $\sigma\circ f(L)=f\circ\sigma(L)$. The permutation interchanges the components of a link $L$ and interchanges the components of the value $f(L)$;
  \item If the component $K_i$ is trivial then $f_i(L)=\tilde 1$;
  \item Let $L=K_1\cup\cdots\cup K_n$, $L'=K_1\cup\cdots\cup K_{n-2}\cup (K_{n-1}\# K_n)$ and $f(L)=(f_i(L))_{i=1}^{n}\in\tilde G^{n}$. Then $f(L')=(f_i(L'))_{i=1}^{n-1}\in\tilde G^{n-1}$ where $f_i(L')=f_i(L)$ for $1\le i\le n-2$ and $f(L')_{n-1}=\widetilde{g_{n-1}g_n}$ for some representatives $g_{n-1},g_n\in G$ of the classes $f_{n-1}(L)$ and $f_n(L)$ respectively.
\end{enumerate}

\begin{thm}\label{thm:concordance invariance}
The map $f$ is a concordance invariant of links.
\end{thm}

\begin{cor}
If $K$ is a slice knot then $f(K)=\tilde 1$.
\end{cor}

\subsection{Concordance of links in the thickened torus}

Let $L=L_1\cup L_2\cup\cdots\cup L_n$ be a link in $S^3$ such that $L_1\cup L_2$ forms the Hopf link. Then $S^3\setminus(L_1\cup L_2)$ is diffeomorphic to the thickening $M=T^2\times (0,1)$ of the torus. Hence, the link $\bar L=L_3\cup\cdots\cup L_n$ can be considered as a link in the thickened torus $M$.

\begin{thm}
Two links $L=L_1\cup L_2\cup\cdots\cup L_n$ and $L'=L'_1\cup L'_2\cup\cdots\cup L'_n$ in $S^3$ such that $L_1\cup L_2$ and $L'_1\cup L'_2$ are Hopf links, are concordant if and only if the links $\bar L$ and $\bar L'$ are concordant in $M$, i.e. there exists a smooth embedding $\bar W\colon \sqcup_{i=3}^n S^1\times [0,1]\to M\times [0,1]$ such that $\bar W\cap M\times\{0\}=\partial_0 \bar W=\bar L\times\{0\}$ and $\bar W\cap M\times\{1\}=\partial_1 \bar W=\bar L'\times\{1\}$.
\end{thm}

\begin{proof}
  After some isotopy of $L'$ we can assume that $L_1=L'_1$ and $L_2=L'_2$.

  Let $\bar L$ and $\bar L'$ are concordant in $M$, i.e. exists a set of cylinders $\bar W\subset M\times[0,1] =S^3\setminus (L_1\cup L_2)\times[0,1]$. 
  Then $W=\bar W\cup L_1\times[0,1] \cup L_2\times[0,1]$ establishes a concordance between $L$ and $L'$ in $S^3$.

 Let $L$ and $L'$ are concordant in $S^3$. Then there exists an embedding $W\colon \sqcup_{i=1}^n S^1\times [0,1]\to S^3\times [0,1]$ so that $W\cap S^3\times\{0\}=\partial_0 W=L\times\{0\}$ and $W\cap S^3\times\{1\}=\partial_1 W=L'\times\{1\}$.

 Assume that $L_i\times\{0\}$ and $L_i\times\{1\}$ belong to one component of $W$, $i=1,2$. Assume also that there exists an isotopy of $S^3\times[0,1]$ which is fixed on the boundary and transforms the components of $W$ containing $L_i\times\{0\}$, $i=1,2$, to $L_i\times[0,1]$. Denote the sufrace $W$ after the isotopy by $W'$.
 Then $\bar W=W'\setminus (L_1\cup L_2)\times[0,1]$ is a concordance of $\bar L$ and $\bar L'$ in $S^3\setminus (L_1\cup L_2)=M$.
\end{proof}

\section{Proof of Theorem~\ref{thm:concordance invariance}}

\begin{proof}
  Consider first the case of knots. Let $K_0$ and $K_1$ be two concordant knots. Then there exists an annulus $W\subset \R^3\times [0,1]$ such that $W\cap \R^3\times\{i\}=K_i$, $i=0,1$.

  Let $t$ be the coordinate on the interval $[0,1]$. We can assume that $t$ is a simple Morse function on $W$. Consider the Reeb graph $T$ of the function $t$ on $W$. That means the graph $T$ is the quotient space $W/\sim$ where $x\sim y$ iff $x,y$ belong to the same connected component of the level $t^{-1}(c)\subset W$ for some $c\in[0,1]$. The graph $T$ is a trivalent graph and its vertices correspond to components of $t^{-1}(c)$ which contains singular points of the map $t$ except the vertices $v_0$, $v_1$ which corresponds to the knots $K_0$ and $K_1$.

  The graph $T$ is a tree because $W$ is an annulus. Then there is a unique path $P$ in $T$ with ends $v_0$ and $v_1$.

  For any edge $e$ in the graph $T$ choose a point $z\in e$. Let $c=t(z)$, $L_c=t^{-1}(c)\subset W$ and $K_i$ be the component of the link $L_c$ that corresponds to $z$. Denote the element $f(L_c)_i\in\tilde G$ by $f(z)$. Since $f$ is an invariant, the value $f(z)$ is the same for all internal points $z\in e$. Hence, there is a well defined map $f\colon E(T)\to \tilde G$ where $E(T)$ is the set of edges of $T$.

  We will prove that $f(e)=\tilde 1$ for any $e\not\in P$ and $f(e)=f(K_0)$ for any $e\in P$.

  Let $e\in E(T)\setminus P$. Let $T(e)$ be the component of the graph $T\setminus e$ which does not contain the path $P$ and $v(e)$ be the end of the edge $e$ in $T_e$. The {\em height} $h(e)$ of the edge $e$ is the number of the edges in $T(e)$.

  Let us prove that $f(e)$ for any $e\in E(T)\setminus P$ by induction on the height $h(e)$.

  Let $h(e)=0$. Then $e$ is a leaf and it corresponds to a trivial component $K_i$. Then $f(e)=\tilde 1$ by the second property of $f$.

  Let $h(e)>0$. Denote the other edges incident to $v(e)$ by $e'$ and $e''$. Then $h(e')<h(e)$ and $h(e'')<h(e)$, hence, $f(e')=f(e'')=\tilde 1$.
  With respect to the function $t$ there are two cases.

  1. Let $e'$ and $e''$ merge to the edge or the edge $e$ split into $e'$ and $e''$. Then the edges $e,e',e''$ correspond to link components $K, K'$ and $K''$ such that $K=K'\# K''$. By the third property of the invariant $f$, there exist elements $g\in f(e)$, $g'\in f(e')$, $g''\in f(e'')$ such that $g=g'g''$. Since $g'=g''=1$, we have $g=1$, hence $f(e)=\tilde 1$.

  2. Let $e$ and $e'$ merge to the edge $e''$ or the edge $e''$ split into $e$ and $e'$. Then there exist $g\in f(e)$, $g'\in f(e')$, $g''\in f(e'')$ such that $g''=gg'$. Since $g'=g''=1$ by the induction, $g=1$ and $f(e)=\tilde 1$.

  Now, let us prove that $f(e)=f(K_0)$ for any $e\in P$. The path $P$ is a sequence of edges $e_0, e_1,\dots, e_n$ where $v_0$ is incident to $e_0$ and $v_1$ is incident to $e_n$. We prove that $f(e_i)=f(K_0)$ by induction on $i$. By definition, $f(e_0)=f(v_0)=f(K_0)$.

  Assume that $f(e_i)=f(K_0)$. There can be two cases. Let the edge $e_i$ split into the edge $e_{i+1}$ and some edge $e\in E(T)\setminus P$. By the third property of the invariant $f$, there exist elements $g_i\in f(e_i)$, $g_{i+1}\in f(e_{i+1})$, $g\in f(e)$ such that $g_i=g_{i+1}g$. Since $e\in E(T)\setminus P$, $f(e)=\tilde 1$ and $g=1$. Then $g_i=g_{i+1}$ and $f(e_{i+1})=f(e_i)=f(K_0)$.

   If the edge $e_i$ merge with some edge $e\in E(T)\setminus P$ to the edge $e_{i+1}$ then $g_{i+1}=g_{i}g$ for some elements $g_i\in f(e_i)$, $g_{i+1}\in f(e_{i+1})$, $g\in f(e)$. Since $e\in E(T)\setminus P$, $f(e)=\tilde 1$ and $g=1$. Then $g_i=g_{i+1}$ and $f(e_{i+1})=f(e_i)=f(K_0)$.

   Thus, we have $f(K_1)=f(v_1)=f(e_n)=f(K_0)$. That means, $f$ is a concordance invariant.

   The concordance invariance for links can be proved analogously.

\end{proof}

\begin{exa}
Consider the RGB-link from~\cite{MP} with parameters $(0,1,0,1,2,1)$, see Fig.~\ref{fig:RGB_link}.

\begin{figure}[h]
\centering
\includegraphics[width=0.5\textwidth]{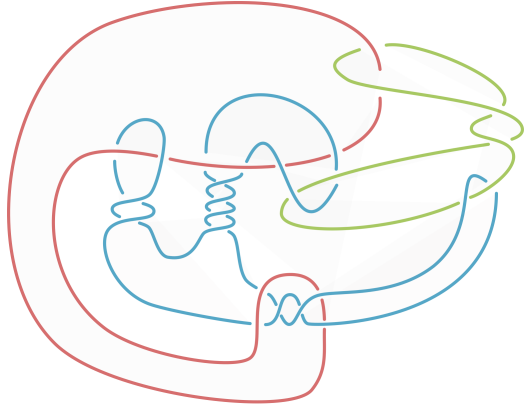}
\caption{The RGB-link}\label{fig:RGB_link}
\end{figure}

The red and the green components of the link form the Hopf link. Then the green and the blue components can be considered as a 2-component link in the thickened cylinder (see Fig.~\ref{fig:RGB_torus} left), and the blue component can be considered as a knot $K$ in the thickened torus (see Fig.~\ref{fig:RGB_torus} right).

\begin{figure}[h]
\centering
\includegraphics[width=0.3\textwidth]{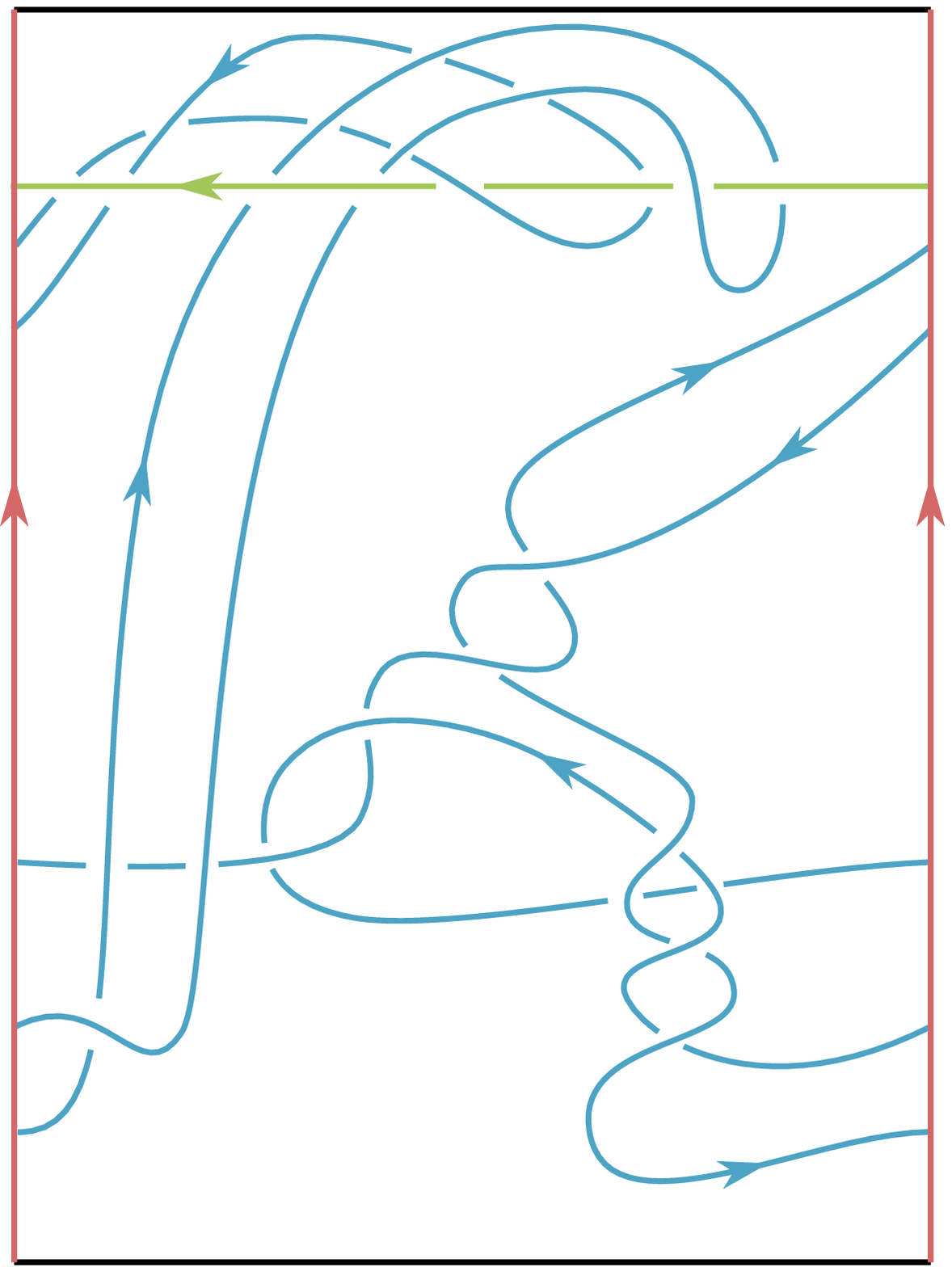}\qquad\qquad\includegraphics[width=0.3\textwidth]{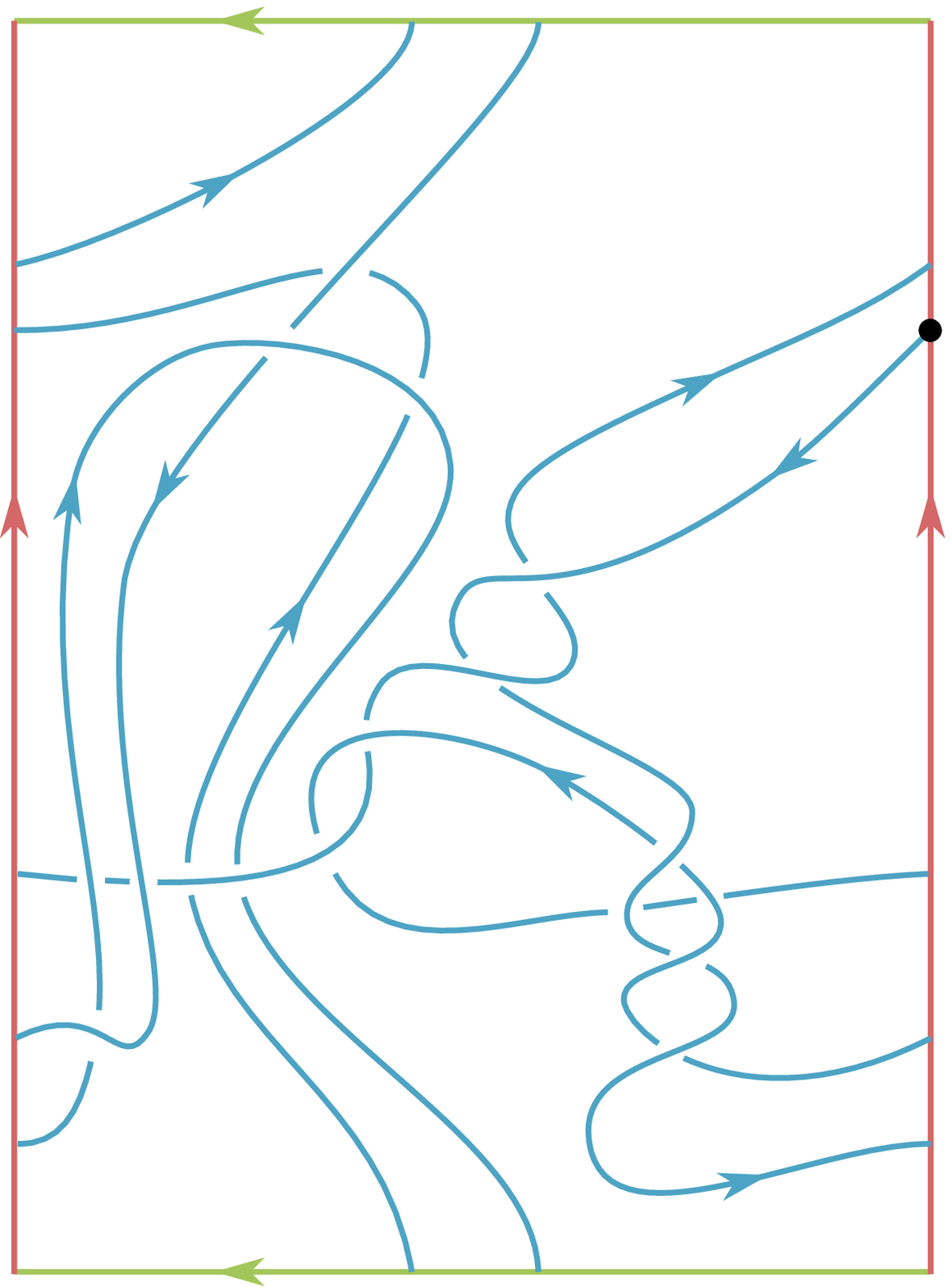}
\caption{The RGB-link defines a GB-link in the cylinder and a B-knot in the torus}\label{fig:RGB_torus}
\end{figure}

Let us calculate the invariant $f'$ of the knot $K$:

\begin{multline*}
f'(K)= bb^{-1}B'(B')^{-1}BB^{-1}BB^{-1}B'aB'B^{-1}b'abb^{-1}bab'aa(b')^{-1}\cdot\\ \cdot ab^{-1}aB^{-1}b'aaB^{-1}Bb^{-1}ab^{-1}=B'B^{-1}B'B^{-1}b'b^{-1}b'b^{-1}.
\end{multline*}

Note that the group $G'$ contains a subgroup isomorphic to $\Z\oplus\Z$ which consists of the elements $(B'B^{-1})^k(b'b^{-1})^l$, $k,l\in\Z$.
The bijection $\phi$ maps the element $(B'B^{-1})^k(b'b^{-1})^l$ to the element $(B'B^{-1})^{-k}(b'b^{-1})^{-l}$. This means that the invariant $f'(K)$ is not trivial in $\tilde G'$.

Thus, the knot $K$ in the thickened torus is not slice.
\end{exa}

\section*{Aknowlegements}

The work of V.O.Manturov was funded by the development program of the Regional Scientific and Educational Mathematical Center of the Volga Federal District, agreement N 075-02-2020.

\end{document}